\documentclass[12pt]{amsproc}
\newtheorem{theorem}{\sc Theorem}
\newtheorem{lemma}[theorem]{\sc Lemma}

\begin{document} 
\author{ Raimundo Bastos}
\address{ Department of Mathematics, University of Brasilia,
Brasilia-DF, 70910-900 Brazil }
\email{bastos@mat.unb.br}
\author{ Pavel Shumyatsky }
\address{ Department of Mathematics, University of Brasilia,
Brasilia-DF, 70910-900 Brazil }
\email{pavel@unb.br}


\thanks{This work was supported by CNPq-Brazil. }

\title[nilpotency]{A Sufficient Condition for Nilpotency of the Commutator Subgroup }
\begin{abstract} Let $G$ be a finite group with the property that if $a,b$ are commutators of coprime orders, then $|ab|=|a||b|$. We show that $G'$ is nilpotent.
\end{abstract}
\maketitle

The following criterion of nilpotency of a finite group was established by B. Baumslag and J. Wiegold \cite{baubau}.
\begin{theorem}\label{bau} Let $G$ be a finite group in which $|ab|=|a||b|$ whenever the elements $a,b$ have coprime orders. Then $G$ is nilpotent.
\end{theorem}

Here the symbol $|x|$ stands for the order of the element $x$ in a group $G$. In the present article we establish a similar criterion of nilpotency for the commutator subgroup $G'$. Recall that an element $g\in G$ is a commutator if $g=x^{-1}y^{-1}xy$ for suitable $x,y\in G$. By definition, the commutator subgroup $G'$ is the subgroup of $G$ generated by all commutators. The purpose of the present article is to prove the following theorem.

\begin{theorem}\label{main} Let $G$ be a finite group in which $|ab|=|a||b|$ whenever the elements $a,b$ are commutators of coprime orders. Then $G'$ is nilpotent.
\end{theorem}

In view of the above theorem one might suspect that a similar phenomenon holds for other group-words. Recall that a group-word $w=w(x_1,\dots,x_s)$ is a nontrivial  element of the free group $F = F(x_1,\dots,x_s)$ on free generators $x_1,\dots,x_s$. A word is a commutator word if it belongs to the commutator subgroup $F'$. Given a group-word $w$, it can be viewed as function defined in any group $G$. The subgroup of $G$ generated by the $w$-values is called the verbal subgroup of $G$ corresponding to the word $w$. It is usually denoted by $w(G)$. One might ask the following question.

 Let $w$ be a commutator word and $G$ a finite group with the property that if $a,b$ are $w$-values of coprime order, then $|ab|=|a||b|$. Is then the verbal subgroup $w(G)$ nilpotent?

Asking a similar question for noncommutator words would not be interesting since an easy counter-example is provided just by any nonabelian simple group $G$, say of exponent $e$, and the word $x^n$, where $n$ is a divisor of $e$ such that $e/n$ is prime.  Even in the case of commutator words the answer to our question is negative: Kassabov and Nikolov showed in \cite{kani} that for any $n\geq7$ the alternating group $A_n$ admits a commutator word all of whose nontrivial values have order 3. We suspect that the answer to the question is positive in the case of multilinear commutator words, that is, words having a form of a multilinear Lie monomial (like for example $[[x_1,[x_2,x_3]],[x_4,x_5]]$).

Throughout the remaining part of this short note $G$ denotes a finite group satisfying the hypothesis of Theorem \ref{main}. We denote by $X$ the set of commutators in $G$. As usual, $\pi(K)$ stands for the set of primes dividing the order of a group $K$. The Fitting subgroup of $K$ is denoted by $F(K)$.

\begin{lemma}\label{bbb} Let $x\in X$ and $N$ be a subgroup normalized by $x$. If $(|x|,|N|)=1$, then $[x,N]=1$.
\end{lemma}
\begin{proof} Choose $y\in N$. The order of the commutator $[x,y]$ is prime to that of $x$. Therefore we must have $|x[x,y]|=|x||[x,y]|$. However $x[x,y]=y^{-1}xy$. This is a conjugate of $x$ and so $|x[x,y]|=|x|$. Therefore $[x,y]=1$.
\end{proof}
\begin{lemma}\label{solu} If $G$ is soluble, then $G'$ is nilpotent.
\end{lemma}
\begin{proof} Arguing by induction on $|G|$ we can assume that the second commutator subgroup $G''$ is nilpotent. Suppose that there are two different primes $p\in\pi(G')$ and $q\in\pi(G'')$. Let $P$ be a Sylow $p$-subgroup of $G'$ and $Q$ a Sylow $q$-subgroup of $G''$. It is straightforward from the Focal Subgroup Theorem \cite[Theorem 7.3.4]{go} that $P$ is generated by $P\cap X$. Lemma \ref{bbb} tells us that if $x\in P\cap X$, we have $[Q,x]=1$. Therefore $[Q,P]=1$. It follows that $P$ is normal in $PG''$ and therefore $P\leq F(PG'')$. Since $PG''$ is normal in $G'$, we conclude that $P\leq F(G)$. This holds for every prime $p$ such that $G''$ is not a $p$-group. Therefore if $G''$ is divisible by at least two different primes, all Sylow subgroups of $G'$ belong to $F(G)$ and so $G'$ is nilpotent. If $G''$ is $q$-group, it is obvious that the Sylow $q$-subgroup of $G'$ is normal in $G$ and so again we conclude that all Sylow subgroups of $G'$ belong to $F(G)$. The proof is complete.
\end{proof}

\begin{proof}[Proof of Theorem \ref{main}]
Now suppose that $G$ is a counter-example of minimal order. So all proper subgroups in $G$ are soluble and we can assume that $G=G'$. Let $R$ be the soluble radical in $G$. It follows that $G/R$ is nonabelian simple. By Lemma \ref{solu} $R'$ is nilpotent. Suppose that $G$ is nonsimple and $R\neq1$ and, for a prime $q$, let $Q$ be the Sylow subgroup of $F(G)$. Let $T$ be the subgroup of $G$ generated by all commutators that are $q'$-elements. In view of the Focal Subgroup Theorem, all Sylow $p$-subgroups of $G'$ for $p\neq q$ are contained in $T$. Therefore the commutator subgroup of $G/T$ is a $q$-group. Since $G=G'$, we conclude that $G=T$. Combining Lemma \ref{bbb} with the Focal Subgroup Theorem we now deduce that $F(G)\leq Z(G)$.

Further, we remark that for any $x\in G$ the subgroup $\langle x,R\rangle$ is soluble and so $\langle x,R\rangle'$ is nilpotent. It follows that $[R,x]\leq F(G)=Z(G)$ and therefore $R=Z_2(G)$. In particular, $R$ is nilpotent and therefore $R=Z(G)$. Thus, our group $G$ is quasisimple.

Since $G$ does not possess a normal 2-complement, it follows from the Frobenius Theorem \cite[Theorem 7.4.5]{go} that $G$ contains a 2-subgroup $H$ and an element of odd order $b\in N_G(H)$ such that $[H,b]\neq1$. In view of Thompson's Theorem \cite[Theorem 5.3.11]{go} we can assume that $H$ is of nilpotency class at most two and $H/Z(H)$ is elementary abelian. We claim that $G$ contains an element $a$ such that 
\newline\newline
\noindent $a$ is a 2-element;\newline
\noindent $a\in X$;\newline
\noindent $a$ has order 2 modulo $Z(G)$.\newline

Indeed, since $H$ is of class at most two, all elements in $[H,g]$ are commutators for any $g\in H$. If for some $g\in H$ the subgroup $[H,g]$ is not contained in $Z(G)$, any element of $[H,g]$ that has order 2 modulo $Z(G)$ has the required properties. Therefore we assume that $[H,g]$ is contained in $Z(G)$ for any $g\in H$. In particular $H'\leq Z(G)$. It follows that $[H,b]\cap C_G(b)\leq Z(G)$ and all elements in $[H,b]$ are commutators modulo $Z(G)$. If $d\in[H,b]$ such that $d\not\in Z(G)$ and $d^2\in Z(G)$, then the commutator $[d,b]$ is as required. This proves the existence of an element $a$ with the above properties.

Now we fix such an element $a$. Since $G/Z(G)$ is nonabelian simple, it follows from the Baer-Suzuki Theorem \cite[Theorem 3.8.2]{go} that there exists an element $t\in G$ such that the order of $[a,t]$ is odd. On the one hand, it is clear that $a$ inverts $[a,t]$. On the other hand, by Lemma \ref{bbb}, $a$ must commute with $[a,t]$. This is a contradiction.
\end{proof}


\begin{thebibliography}{10}
\bibitem{baubau} B. Baumslag and J. Wiegold, {\it A Sufficient Condition for Nilpotency in a Finite Group}, preprint available at arXiv:1411.2877v1 [math.GR]. 

\bibitem{go} D. Gorenstein, \emph{Finite Groups}, Chelsea Publishing Company, New York, 1980

\bibitem{kani} M. Kassabov and N. Nikolov, {\it Words with few values in finite simple groups},  The Quarterly Journal of Mathematics, {\bf 64}, 2013, 1161--1166.
\end{thebibliography}
\end{document}